\documentclass[a4paper,12pt]{article}

\usepackage{amsfonts,amssymb,amsmath,amsthm,latexsym,epsfig,amscd,bbm,stmaryrd}
\usepackage{mathrsfs}
\usepackage{graphicx,graphics,psfrag,epsf}

\usepackage[english]{babel}
\makeatletter \@addtoreset{equation}{section}
\makeatother

\makeatletter \@addtoreset{enunciato}{section}
\makeatother

\newcounter{enunciato}[section]

\newtheorem{ittheorem}{Theorem}
\newtheorem{itlemma}{Lemma}
\newtheorem{itproposition}{Proposition}
\newtheorem{itdefinition}{Definition}
\newtheorem{itremark}{Remark}
\newtheorem{itclaim}{Claim}
\newtheorem{itfact}{Fact}
\newtheorem{itconjecture}{Conjecture}
\newtheorem{itcorollary}{Corollary}

\newenvironment{theorem}{\addtocounter{enunciato}{1}
\begin{ittheorem}}{\end{ittheorem}}

\newenvironment{lemma}{\addtocounter{enunciato}{1}
\begin{itlemma}}{\end{itlemma}}

\newenvironment{proposition}{\addtocounter{enunciato}{1}
\begin{itproposition}}{\end{itproposition}}

\newenvironment{definition}{\addtocounter{enunciato}{1}
\begin{itdefinition}}{\end{itdefinition}}

\newenvironment{remark}{\addtocounter{enunciato}{1}
\begin{itremark}}{\end{itremark}}

\newenvironment{conjecture}{\addtocounter{enunciato}{1}
\begin{itconjecture}}{\end{itconjecture}}

\newenvironment{corollary}{\addtocounter{enunciato}{1}
\begin{itcorollary}}{\end{itcorollary}}

\newcommand{\be}[1]{\begin{equation}\label{#1}}
\newcommand{\ee}{\end{equation}}

\newcommand{\bl}[1]{\begin{lemma}\label{#1}}
\newcommand{\el}{\end{lemma}}

\newcommand{\br}[1]{\begin{remark}\label{#1}}
\newcommand{\er}{\end{remark}}

\newcommand{\bt}[1]{\begin{theorem}\label{#1}}
\newcommand{\et}{\end{theorem}}

\newcommand{\bd}[1]{\begin{definition}\label{#1}}
\newcommand{\ed}{\end{definition}}

\newcommand{\bp}[1]{\begin{proposition}\label{#1}}
\newcommand{\ep}{\end{proposition}}

\newcommand{\bc}[1]{\begin{corollary}\label{#1}}
\newcommand{\ec}{\end{corollary}}

\newcommand{\bcj}[1]{\begin{conjecture}\label{#1}}
\newcommand{\ecj}{\end{conjecture}}

\newcommand{\bpr}{\begin{proof}}
\newcommand{\epr}{\end{proof}}

\def \Z {{\mathbb Z}}
\def \R {{\mathbb R}}

\def \N {{\mathbb N}}

\def \ba {\begin{array}}
\def \ea {\end{array}}

\begin{document}

%%%%%%%%%%%%%%%%% TITLE PAGE %%%%%%%%%%%%%%%%%%%%%%%%%%%%%%%%%%%%%

\title{Non-trivial linear bounds for a random walk driven by a simple symmetric exclusion process}

\author{\renewcommand{\thefootnote}{\arabic{footnote}}
R.S.\ dos Santos \footnotemark[1]}

\footnotetext[1]{
Mathematical Institute, Leiden University, P.O.\ Box 9512,
2300 RA Leiden, The Netherlands}

\maketitle

\begin{abstract}
Non-trivial linear bounds are obtained for the displacement of a random walk in a dynamic random environment given
by a one-dimensional simple symmetric exclusion process in equilibrium.
The proof uses an adaptation of multiscale renormalization methods of
Kesten and Sidoravicius \cite{KeSi03}.

\vspace{0.5cm}\noindent
\end{abstract}

%%%%%%%%%%%%%%%%%%%%% SECTION 1 %%%%%%%%%%%%%%%%%%%%%%%%%%%%%%%%%%%%%%%%

\section{Introduction, results and motivation}
\label{sec:intro}

\subsection{Setup}
\label{subsec:setup}

In this note, we discuss linear scaling properties of a random walk in a dynamic random environment (RWDRE),
where the role of the random environment is taken by a one-dimensional simple symmetric exclusion process (SSEP).
The latter is the c\`adl\`ag Markov process $\xi =(\xi_t)_{t\ge 0}$ with state space $E := \{0,1\}^{\Z}$ whose
infinitesimal generator $\mathcal{L}$ acts on bounded local functions $f$ in the following manner:
\begin{equation}\label{defgeneratorSSEP}
\left(\mathcal{L}f\right)(\eta) := \sum_{x \in \Z} f(\eta^{x,x+1}) -f(\eta)
\end{equation}
where $\eta \in \{0,1\}^{\Z}$ and $\eta^{x,y}$ is defined by
\begin{equation}\label{defetaxy}
\eta^{x,y}(z) = \left\{ 
\begin{array}{rl}
\eta(x) & \text{ if } z = y;\\ 
\eta(y) & \text{ if } z = x;\\
\eta(z) & \text{ otherwise}.\\
\end{array}\right.
\end{equation}
For a detailed description, we refer the reader to Liggett \cite{Li85}, Chapter VIII.
We say that the site $x$ is occupied by a \emph{particle} 
at time $t$ if $\xi_t(x)=1$ and is \emph{vacant} (alternatively, occupied by a \emph{hole}) if $\xi_t(0) = 0$.

For a fixed realization of $\xi$, the random walk in dynamic random environment
$W = (W_t)_{t\ge0}$  is the time-inhomogeneous Markov process that starts at $0$ and,
given that $W_t = x$, jumps to
\begin{equation}\label{defratesW}
\begin{array}{rcl}
x+1 & \text{ with rate } & \alpha_1 \xi_t(x) + \alpha_0 \left[1 -\xi_t(x) \right],\\
x-1 & \text{ with rate } & \beta_1 \xi_t(x) + \beta_0 \left[1 -\xi_t(x) \right],
\end{array}
\end{equation}
where $\alpha_i, \beta_i \in (0,\infty)$, $i \in \{0,1\}$. We will assume that
\begin{equation}\label{constanttotalrates}
\alpha_0 + \beta_0 = \alpha_1+\beta_1 =: \gamma
\end{equation}
and
\begin{equation}\label{constanttotalrates}
v_1 > v_0 \text{ with } v_0 := \alpha_0 - \beta_0  \text{ and } v_1 := \alpha_1 - \beta_1,
\end{equation}
i.e., the total jump rate is constant and equal to $\gamma$,
and the local drift is larger on particles than on holes.
The latter is made w.l.o.g., since the SSEP is invariant
under reflection through the origin. 
We will denote by $\mathbb{P}_{\eta}$ the joint law of $W$ and $\xi$ when $\xi_0 = \eta$.
We will draw $\xi_0$ from a Bernoulli
product measure $\nu_\rho$ with $\rho \in (0,1)$; these are known to be the only
non-trivial extremal invariant measures for the SSEP.

While many results for RWDRE have been obtained in the past few years 
for random environments exhibiting uniform and fast enough mixing (see e.g.\ 
Avena \cite{Avthesis}, and dos Santos \cite{dSathesis}), 
very little is known when the random environment mixes in a non-uniform
way, as happens in the SSEP. For example, there are still no general laws of large numbers available 
for such cases. In particular, for the model described here, the law of large numbers has 
only been proven under the restriction that $v_1>v_0 > 1$ (see Avena, dos Santos and V\"ollering \cite{AvdSaVo12}). 
Another recent result is the paper by den Hollander, Kesten and Sidoravicius \cite{dHoKeSipr}, 
where an approximate law of large numbers is proven when the random environment 
is a high-density Poisson field of independent random walks.

\subsection{Main result}
\label{subsec:mainresult}

It is easy to see, with a coupling argument,
that $W$ lies between two homogeneous random walks with drifts $v_0$ and $v_1$.
In particular, any subsequential limit of $t^{-1}W_{t}$ as $t \to \infty$ lies in the interval $[v_0,v_1]$.
But would it be possible, even along a subsequence, for $W$ to travel at one of the extremal speeds?
For the case of the SSEP, the following theorem answers this question in the negative.

\begin{theorem}\label{mainthm}
For any $\rho \in (0,1)$, there exist $v_-, v_+ \in (v_0,v_1)$ such that
\begin{equation}\label{mainthmeq}
v_- \le \liminf_{t \to \infty}t^{-1}W_t \le \limsup_{t \to \infty} t^{-1}W_t \le v_+ \quad \mathbb{P}_{\nu_\rho} \text{-a.s.}
\end{equation} 
\end{theorem}

While this result is ``intuitively obvious'', it does not seem a trivial fact to prove. 
For dynamic random environments consisting of single-site
spin-flips with bounded flip rates, there is a simple proof strategy
since particles and holes can be found locally ``around'' the random walk.
For the supercritical contact process, the proof
by den Hollander and dos Santos \cite{dHodSapr} that $W$ cannot travel with speed $v_0$ is already non-trivial and relies 
on model-specific features. The proof of Theorem~\ref{mainthm} given
here is based on the \emph{multiscale analysis} scheme put forth by Kesten in Sidoravicius \cite{KeSi03}, 
and seems exceedingly heavy for such a simple fact.
It has however the advantage of being easier to generalize;
while several technical facts are verified here
specifically for the SSEP, the overall proof
strategy should work in much greater generality. 
For example, the analogous result for the supercritical contact process
can be reobtained with this approach.

\subsection{Essential enhancements}
\label{sec:essenh}

Our question can also be formulated in terms of \emph{essential enhancements}, 
in analogy with percolation theory (see e.g.\ Grimmett \cite{Gr99}, Chapter 3).
From this point of view, $W$ is seen as a perturbation of a homogeneous random walk with drift $v_0$,
and $\rho$ is the intensity of the perturbation.
The question then becomes: is this perturbation, for any $\rho >0$, an ``essential enhancement''
in the sense that it changes the linear scaling of $W$?

Let us look at what can happen for random walks in \emph{static} 
one-dimensional random environments.
For these models, there are criteria for recurrence/transience
as well as laws of large numbers proven under very general assumptions (see e.g.\ Zeitouni \cite{Ze04}).
If $v_0 = 0 < v_1$, then the random walk is always transient to the right in any ergodic
random environment with a positive density of particles.
But random walks in static random environments can exhibit \emph{slow-down} phenomena;
for example, there are regimes where the random walk can be transient to the right
with zero speed. 
In the case of i.i.d.\ static random environments, the latter can only happen when $v_0 < 0 < v_1$.
Therefore, if $v_0 = 0 < v_1$, then the perturbation given by a static i.i.d.\
random environment is always an essential enhancement, as long as the density of $1$'s is positive.

Consider, however, the following example of a stationary and ergodic static random environment 
with positive particle density that does \emph{not} result in an essential enhancement. 
Let $L$ be an $\N$-valued random variable with finite first moment but infinite second moment.
Partition $\Z$ into intervals in a translation-invariant way such that the length of
each interval is independent and distributed as $L$. Let $\eta$ be obtained by coloring
each interval with $1$'s or $0$'s according to independent fair coin tosses.
On top of this static random environment, put a random walk with $\beta_0 = \alpha_0 = 1/2$,
$\beta_1 = 0$ and $\alpha_1 = 1$. 
As discussed above, this random walk is transient to the right; 
therefore, it eventually reaches a point where there is an 
interval full of $1$'s to its left (into which it cannot backtrack) 
and an interval to its right whose law is still independent of the past. 
In other words, the times when $W$ crosses the boundary between an interval full of $1$'s 
and the next interval are \emph{regeneration times}.
This observation allows us to estimate the speed of $W$ by a constant times the ratio between the expectation of $L$ and the expected time
required by $W$ to cross one interval, given that the interval to the left is occupied. 
The latter turns out to be infinite, so that $W$ has speed $0 = v_0$.
Therefore, in this example the random environment is not an essential enhancement, 
despite having particle density equal to $1/2$.

\subsection{Outline}
\label{subsec:outline}

The rest of the paper is organized as follows.
In Section~\ref{sec:construction}, we construct particular versions of the SSEP and of the random walk.
In Section~\ref{sec:proofmainthm}, we give the proof of Theorem~\ref{mainthm} with
the help of a proposition (Proposition~\ref{mainprop} below) concerning
\emph{rarefied} and \emph{turbulent} regions in the SSEP.
In Section~\ref{sec:PPS}, we lay out the basic tools that will be used 
to prove Proposition~\ref{mainprop} in Section~\ref{sec:proofmainprop}, 
where all constructions and estimates specific to the SSEP are carried out.

\section{Construction of the model}
\label{sec:construction}

In Section~\ref{subsec:constructionSSEP} we construct the SSEP and, in
Section~\ref{subsec:constructionRW}, the random walk on top of the SSEP.

\subsection{Graphical construction of the SSEP}
\label{subsec:constructionSSEP}

It will be convenient to have a graphical construction of the SSEP including negative times.
Let $\mathcal{E}$ be the set of edges of $\Z$, i.e., all unordered pairs of neighbouring sites, 
and let $\mathcal{A} = \left(\mathcal{A}_e \right)_{e \in \mathcal{E}}$ be a collection of independent Poisson point processes on $\R$
with intensity $1$. Draw each event of $\mathcal{A}_e$ in space-time as an arrow between the two sites
connected by $e$. This gives rise to a system of random paths in $\Z \times \R$ as follows.
For each $(x,t) \in \Z \times \R$, there exists a.s.\ a unique doubly infinite right-continuous path that goes either vertically in time
or (forcibly) across arrows of $\mathcal{A}$. For $s \in \R$, let $\zeta^t_s(x)$ denote the position of this path at time $s$.

%%%%%%%%%%%%%%%%% FIGURE GRAPH REP %%%%%%%%%%%%%%%%%%%%%%%%%%%%
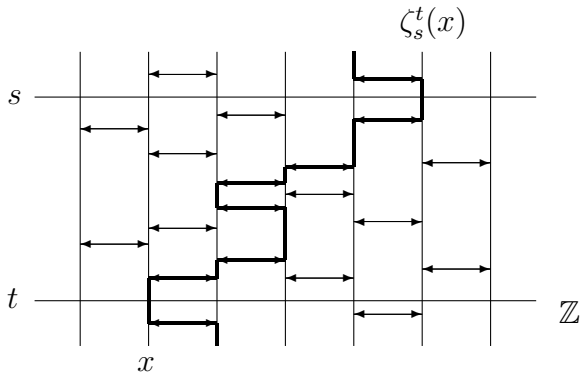
\begin{figure}[hbtp]
\vspace{1cm}
\begin{center}
\setlength{\unitlength}{0.3cm}
\begin{picture}(20,10)(0,0)
%horizontal lines
\put(0,1){\line(22,0){22}} \put(0,10){\line(22,0){22}}

%vertical lines
\put(2,-1){\line(0,1){13}}
\put(5,-1){\line(0,1){13}} \put(8,-1){\line(0,1){13}}
\put(11,-1){\line(0,1){13}} \put(14,-1){\line(0,1){13}}
\put(17,-1){\line(0,1){13}} \put(20,-1){\line(0,1){13}}

%path
{\linethickness{0.05cm}
\put(8,-1){\line(0,1){1}} \put(8,0){\line(-1,0){3}}
\put(5,0){\line(0,1){2}} \put(5,2){\line(1,0){3}}
\put(8,2){\line(0,1){0.8}} \put(8,2.8){\line(1,0){3}}
\put(11,2.8){\line(0,1){2.3}} \put(11,5.1){\line(-1,0){3}}
\put(8,5.1){\line(0,1){1.1}} \put(8,6.2){\line(1,0){3}}
\put(11,6.2){\line(0,1){0.7}} \put(11,6.9){\line(1,0){3}}
\put(14,6.9){\line(0,1){2.1}} \put(14,9){\line(1,0){3}}
\put(17,9){\line(0,1){1.8}} \put(17,10.8){\line(-1,0){3}}
\put(14,10.8){\line(0,1){1.2}}
}

%arrows on the path
\put(5,0){\vector(1,0){3}}\put(8,0){\vector(-1,0){3}}
\put(5,2){\vector(1,0){3}}\put(8,2){\vector(-1,0){3}}
\put(8,2.8){\vector(1,0){3}}\put(11,2.8){\vector(-1,0){3}}
\put(8,5.1){\vector(1,0){3}}\put(11,5.1){\vector(-1,0){3}}
\put(8,6.2){\vector(1,0){3}}\put(11,6.2){\vector(-1,0){3}}
\put(11,6.9){\vector(1,0){3}}\put(14,6.9){\vector(-1,0){3}}
\put(14,9){\vector(1,0){3}}\put(17,9){\vector(-1,0){3}}
\put(14,10.8){\vector(1,0){3}}\put(17,10.8){\vector(-1,0){3}}

%arrows off the path
\put(2,3.5){\vector(1,0){3}}\put(5,3.5){\vector(-1,0){3}}
\put(2,8.6){\vector(1,0){3}}\put(5,8.6){\vector(-1,0){3}}
\put(5,4.2){\vector(1,0){3}}\put(8,4.2){\vector(-1,0){3}}
\put(5,7.5){\vector(1,0){3}}\put(8,7.5){\vector(-1,0){3}}
\put(5,11){\vector(1,0){3}}\put(8,11){\vector(-1,0){3}}
\put(8,9.2){\vector(1,0){3}}\put(11,9.2){\vector(-1,0){3}}
\put(11,2){\vector(1,0){3}}\put(14,2){\vector(-1,0){3}}
\put(11,5.7){\vector(1,0){3}}\put(14,5.7){\vector(-1,0){3}}
\put(14,0.4){\vector(1,0){3}}\put(17,0.4){\vector(-1,0){3}}
\put(14,4.5){\vector(1,0){3}}\put(17,4.5){\vector(-1,0){3}}
\put(17,2.4){\vector(1,0){3}}\put(20,2.4){\vector(-1,0){3}}
\put(17,7.1){\vector(1,0){3}}\put(20,7.1){\vector(-1,0){3}}

%symbols
\put(16,12.9){$\zeta_s^t(x)$} \put(4.5,-2.1){$x$}
\put(-1.2,0.7){$t$}
\put(-1.2,9.7){$s$}
\put(23,0){$\mathbb{Z}$}
%\put(11,0){\circle*{.35}}
%\put(8,11){\circle*{.35}}
\end{picture}
\end{center}
\caption{\small Graphical representation. The arrows represent
events of $\mathcal{A}$. The thick lines mark
the path $\zeta_s^t(x)$.} \label{graphrep}
\end{figure}
%%%%%%%%%%%%%%%%%%%%%%%%%%%%%%%%%%%%%%%%%%%%%%%%%%%%%%%%%%%%%%%%%%%%%%%%%%%%%%

Given $\eta \in \{0,1\}^{\Z}$, we will define the SSEP $\xi = (\xi_t)_{t \in \R}$ by
\begin{equation}\label{defSSEP}
\xi_t(x) := \eta(\zeta^t_0(x)),
\end{equation}
i.e., a space-time point $(x,t)$ is occupied if and only if the path going through it hits an occupied site at time $0$.
If we take $\eta$ to be distributed as $\nu_\rho$, $\rho \in (0,1)$, then we may check that this construction indeed results in
a stationary process with the correct distribution. 
To verify this, we only need to note that $\xi_t(x) = \xi_s(\zeta^t_s(x))$ for any $s,t \in \R$  
and that, by the product structure and exchangeability of $\nu_\rho$, $\xi_s$ is independent of $\mathcal{A}$.

\subsection{The random walk on top of the SSEP}
\label{subsec:constructionRW}
We next give a particular construction of the random walk model described in the introduction.
Take a Poisson process $N = (N_t)_{t \ge 0}$ with rate $\gamma$
and two sequences $J^1=(J^1_k)_{k \in \N}$ and $J^0=(J^0_k)_{k \in \N}$
of i.i.d.\ $\{-1,+1\}$-valued random variables taking the value $+1$ with probability $\alpha_1/\gamma$ and $\alpha_0/\gamma$
, respectively. These random variables are taken such that $\xi,N,J^1,J^0$ are jointly independent.

The random walk $W$ is a functional of $(\xi,N,J^1,J^0)$ obtained as follows.
We set $W_0 :=0$. At a time $t>0$, $W$ jumps if and only if $N$ jumps, 
and the increment is given by $W_t -W_{t-} = J^i_{N_t}$,
where $i = \xi_t(W_{t-})$ is the state of the exclusion process 
at the position of $W$ just before the jump.

Setting
\begin{equation}\label{defjumpsonpartholes}
\begin{array}{rcl}
N^{1}_t & := & \# \{t \in [0,t] \colon\, W_t \neq W_{t-} \text{ and } \xi_{t}(W_{t-}) = 1 \}, \\
N^{0}_t & := & \# \{t \in [0,t] \colon\, W_t \neq W_{t-} \text{ and } \xi_{t}(W_{t-}) = 0\}, \\
\end{array}
\end{equation}
then $N^0_t + N^1_t = N_t$ and we see that $W$ has the following representation:
\begin{equation}\label{representationRW}
W_t = S^{1}_{N^{1}_t} + S^{0}_{N^{0}_t}
\end{equation}
where $(S^{i}_n)_{n \in \N_0}$, $i \in \{0,1\}$, are discrete-time simple random walks that jump to the right with probability $\alpha_i/\gamma$. 
From this we immediately get
\begin{equation}\label{speedvsdensity}
\begin{array}{rcl}
\liminf_{t \to \infty} t^{-1}W_t & = & v_0 + (v_1 - v_0) \liminf_{t \to \infty} (\gamma t)^{-1}N^1_t,\\
\limsup_{t \to \infty} t^{-1}W_t & = & v_1 - (v_1 - v_0) \liminf_{t \to \infty} (\gamma t)^{-1}N^0_t.
\end{array}
\end{equation}

\section{Proof of Theorem~\ref{mainthm}}
\label{sec:proofmainthm}

Since the \emph{holes} of a SSEP under $\mathbb{P}_{\nu_\rho}$
have the same distribution as the \emph{particles} of a SSEP under $\mathbb{P}_{\nu_{1-\rho}}$,
we may w.l.o.g.\ restrict ourselves to proving the statement for the $\liminf$ in \eqref{mainthmeq}.

The main idea in the proof of Theorem~\ref{mainthm} is that, because the jump rates are positive and bounded, 
the random walk can spend time on top of particles whenever
it is in a region of the environment that is not too rough, namely, neither too rarefied nor too turbulent.
A \emph{rarefied} region is one where the density of the environment is atypically low. 
A \emph{turbulent} region is one where the environment is moving atypically fast.
It is of course not possible to control such deviations of the environment in all space and time simultaneously,
but, as we will see in Proposition~\ref{mainprop} below, it is possible to show that, 
in most of the regions \emph{accessible} to the random walk, the environment 
cannot deviate too much from its typical behaviour.

In Section~\ref{subsec:rareturb} we state Proposition~\ref{mainprop}. 
In Section~\ref{subsec:proofmainthm}, we use this proposition to prove Theorem~\ref{mainthm}. 
The proof of Proposition~\ref{mainprop} is given in Section~\ref{sec:proofmainprop}.

\subsection{Rarefied and turbulent regions}
\label{subsec:rareturb}

For $r \in \N$, let $\omega_r \le \Delta_r \in \N$ and $\rho_r, \epsilon_r \in (0,1)$ be given parameters.
Let
\begin{equation}\label{defrblocks} 
B_r(k,s) := [k,k+\Delta_r)\times[s,s+\Delta_r), \quad k,s \in \Delta_r\Z,
\end{equation}
be blocks in $\R^d$ with side length $\Delta_r$, called $r$-blocks.
For $x \in \Z$ and $t \in \R$, we write
\begin{equation}\label{defsum}
 \Sigma_r^x(\xi_t) := \sum_{y \in [x,x+\omega_r)} \xi_t(y)
\end{equation}
to denote the number of particles present in $[x,x+\omega_r)$ at time $t$.
We call a set $A \subset \R^2$ $r$-\emph{rarefied} if there exists
 $(x,t) \in \Z^2$ with $[x,x+\omega_r)\times\{t\} \subset A$ and such that
 $\Sigma_r^x(\xi_t) < \rho_r \omega_r$. We call $A$ $r$-\emph{turbulent} if there exists $(x,t) \in A \cap \Z^2$ 
and $s \in (0,\epsilon_r)$ such that $\xi_{t+s}(x) \neq \xi_t(x)$.

For $\ell \in (0,\infty)$, let
\begin{align}\label{defcurlyW1}
\mathcal{W}_\ell := \{\text{all paths } & \text{in } \R^2 \text{ starting at } 0 \text{ which are continuous, }\nonumber \\ 
& \;\; \text{piecewise }C^1, \text{ and have length at most } \ell\},
\end{align}
and put
\begin{equation}\label{defPhis}
\begin{array}{rcl}
\Phi^{\mathrm{r}}_r(\ell) & := & \sup_{w \in \mathcal{W}_{\ell}} \#\{r\text{-rarefied } r\text{-blocks intersected by } w \},\\
\Phi^{\mathrm{t}}_r(\ell) & := & \sup_{w \in \mathcal{W}_{\ell}} \#\{r\text{-turbulent } r\text{-blocks intersected by } w \}.
\end{array}
\end{equation}
The key ingredient in the proof of Theorem~\ref{mainthm} is the following proposition.
\begin{proposition}\label{mainprop}
For any $\rho \in (0,1)$, 
there exist $(\Delta_r,\omega_r,\rho_r,\epsilon_r)_{r\in\N}$
as above such that, $\mathbb{P}_{\nu_\rho}$-a.s.,
\begin{equation}\label{mainpropeq}
\begin{array}{rl}
\text{(a) } & \lim_{r \to \infty} \limsup_{\ell \to \infty} \ell^{-1} \Delta^2_r\Phi^{\mathrm{r}}_r(\ell) = 0,\\
\text{(b) } & \lim_{r \to \infty} \limsup_{\ell \to \infty} \ell^{-1} \Delta^2_r\Phi^{\mathrm{t}}_r(\ell) = 0.
\end{array}
\end{equation}
\end{proposition}

Part (a) will be proved using a \emph{multiscale renormalization scheme}
developed by Kesten and Sidoravicius (see \cite{KeSi03}; we also borrow some ideas from \cite{KeSi05}).
The adaptation is straightforward, and some simplifications are possible in our setting. 
Nevertheless, for completeness, we include all the details.
The main new ingredient is a comparison between the SSEP and a system of independent random walkers, which is due to Liggett.
The proof of part (b) uses a similar strategy, but is much simpler.

To simplify the exposition, we present the proof in dimension one only. 
There are no technical issues to extend it to higher dimensions. Small complications
arise in the proof of Lemma~\ref{lemma:stochdoma} below, but they can be dealt with straightforwardly.

\subsection{Proof of Theorem~\ref{mainthm}}
\label{subsec:proofmainthm}
\begin{proof}
Fix $\rho \in (0,1)$ and recall the definition of $N^1$ in \eqref{defjumpsonpartholes}.
By \eqref{speedvsdensity}, it is enough to prove the existence of a $\delta_0 > 0$ such that
\begin{equation}\label{pmainthmeq2}
\liminf_{t \to \infty}t^{-1}N^{1}_t \ge \delta_0 \;\; \mathbb{P}_{\nu_\rho}\text{-a.s.}
\end{equation}

Regard $(W_s)_{s \in[0,t]}$ as a path in $\R^2$ and denote its length by $\ell_t = t + N_t$. 
Recall that $N$ is a Poisson process with rate $\gamma > 0$, independent of $\xi$.
Using Proposition~\ref{mainprop}, fix $\ell_* \in (1,\infty)$ and $r_* \in \N$ such that
\begin{equation}\label{pmainthmeq4}
\Delta^2_{r_*} \left\{ \Phi^{\mathrm{r}}_{r_*}(\ell) + \Phi^{\mathrm{t}}_{r_*}(\ell) \right\} \le \frac{\ell}{2(1+\gamma)} \;\;\; \mathbb{P}_{\nu_\rho}\text{-a.s. } \forall \; \ell \ge \ell_*.
\end{equation}

Let $B^*_t(W)$ be the unique $r_*$-block containing the spacetime point $(W_t,t)$.
We call $B_t^*(W)$ \emph{rough} if it is either $r_*$-rarified or turbulent, and call it \emph{smooth} otherwise.
For $t \ge 0$, let
\begin{equation}\label{pmainthmeq5}
\Theta^*_t(W):= \sum_{s=0}^{\lfloor t \rfloor} \mathbbm{1}_{ \{B_s^*(W) \text{ is rough}\}}
\end{equation}
denote the total number of integer times between $0$ and $t$ at which $W$ is inside a rough block. 
Since $W$ can spend at most $\Delta_{r_*}$ time units in each rough block, 
if $t \ge \ell_*$, then by \eqref{pmainthmeq4} we have
\begin{align}\label{pmainthmeq6}
\Theta^*_t(W) 
& \le \Delta_{r_*} \left\{ \Phi^{\mathrm{r}}_{r_*}(\ell_t) + \Phi^{\mathrm{t}}_{r_*}(\ell_t) \right\}\nonumber\\
& \le \frac{1}{\Delta_{r_*}}\frac{\ell_t}{2(1+\gamma)} \le \frac{\ell_t}{2(1 + \gamma)} \quad \; \mathbb{P}_{\nu_\rho}\text{-a.s.}
\end{align}

For $s \in \N_0$, let
\begin{equation}\label{pmainthmeq7}
Y_{s+1} := \mathbbm{1}_{\{N^{1}_{s+1}>N^{1}_s\}}.
\end{equation}
Note that $N^{1}_{s+1} > N^{1}_s$ if and only if 
$W$ jumps at least once from a particle in the time interval $(s,s+1]$.
Since $W$ has uniformly positive jump rates, 
for any $s \ge0$, $r \in \N$, $\epsilon > 0$ and $j \in [W_s-r,W_s+r]$,
\begin{equation}\label{pmainthmeq8}
\mathbb{P}_{\nu_\rho} \left( W \text{ jumps once from } j \text{ in the time interval } (s,s+ \epsilon)\mid (W_u)_{u \in [0,s]}, \xi \right)
\ge \delta
\end{equation}
for some $\delta = \delta(r,\epsilon)>0$. Therefore, if $B^*_s(W)$ is smooth, then
\begin{equation}\label{pmainthmeq9}
\mathbb{P}_{\nu_\rho}\left(Y_{s+1} = 1 \mid (W_u)_{u \in [0,s]}, \xi \right) \ge \delta_* := \delta(r_*, \epsilon_{r_*})
\end{equation}
since there is at time $s$ at least one particle in $[W_s - r_*, W_s+ r_*]$ 
that does not move before time $s+\epsilon_{r_*}$. Therefore we can couple $Y$ with 
an i.i.d.\ sequence $(\tilde{Y}_{s})_{s \in \N}$ of Bernoulli($\delta_*$) random variables such that $Y_{s+1} \ge \tilde{Y}_{s+1}$
if $B^*_{s}(W)$ is smooth. 
%On the other hand, by \eqref{pmainthmeq6}, the number of integer times $s$ in $[0,t-1]$ for which $B^*_{s}(W)$ is smooth is asymptotically %at least $t/2$; thus \eqref{pmainthmeq2} holds with $\delta_0 = \delta_*/2$.

Using these observations, we can write, for $t  \ge \ell_*$,
\begin{align}\label{pmainthmeq10}
t^{-1}N^1_t
& \ge t^{-1} \sum_{s=1}^{\lfloor t \rfloor}Y_{s} 
 \ge \;\; t^{-1} \sum_{\stackrel{s \in [1,t]\cap \N \colon}{\small{B^*_{s-1}(W)} \small{\text{ is smooth}}}} \tilde{Y}_{s} \nonumber\\
& \ge \left(\frac{\lfloor t \rfloor - \Theta^*_t(W)}{t}\right)\; \#\left\{\stackrel{s \in [1,t]\cap \N \colon}{B^*_{s-1}(W) \tiny{\text{ is smooth}}}\right\}^{-1}
\sum_{\stackrel{s \in [1,t]\cap \N \colon}{B^*_{s-1}(W) \text{ is smooth}}} \tilde{Y}_{s}.
\end{align}
By \eqref{pmainthmeq6}, the $\liminf$ as $t \to \infty$ of the term in parentheses in the r.h.s.\ of \eqref{pmainthmeq10}
is at least $1/2$. The remaining term converges to $\delta_*$, since the number of integer times $s$ in $[1,t]$ for which $B_{s-1}^*(W)$ is smooth is unbounded.
Thus \eqref{pmainthmeq2} holds with $\delta_0 = \delta_*/2$.
\end{proof}

\section{Block percolation and partitioned systems}
\label{sec:PPS}

In this section we present a percolation result, Proposition~\ref{prop:PPSseq} below,
which will play an important role in the proof of Proposition~\ref{mainprop} in Section~\ref{sec:proofmainprop}.

\subsection{Percolative systems}
\label{subsec:percsys}

Fix $d \in \N \setminus \{1\}$ and $\Delta \in (0,\infty)$. 
For $k = (k_1,\ldots,k_d) \in \Delta \Z^d$, let
\begin{equation}\label{defB}
B_{\Delta}(k) := \prod_{i=1}^{d}[k_i, k_i + \Delta)
\end{equation}
be the block in $\R^d$ of side length $\Delta$ with lower-left corner at $k$.
A collection of random variables
\begin{equation}\label{defPS}
\Upsilon = (\Upsilon(k))_{k\in \Delta \Z^d }, \;\; \Upsilon(k) \in \{0,1\} \text{ for each } k \in \Delta\Z^d,
\end{equation}
is called call a \emph{percolative system} (PS) with scale $\Delta$.
We interpret $\Upsilon$ by saying that a block $B_\Delta(k)$ is \emph{open} if $\Upsilon(k)=1$, 
and \emph{closed} otherwise. See Figure~\ref{fig:blockperc}.

We aim to bound the number of open blocks that intersect paths of a certain fixed length in $\R^d$.
For $\ell \in (0,\infty)$, let, analogously to \eqref{defcurlyW1},
\begin{align}\label{defcurlyW}
\mathcal{W}_\ell := \{\text{all paths in } & \R^d \text{ starting at } 0 \text{ which are continuous, }\nonumber \\ 
& \text{ piecewise }C^1, \text{ and have length at most } \ell\}.
\end{align}

%%%%%%%%%%%%%%%%%%%%% FIGURE BLOCK PERCOLATION %%%%%%%%%%%%%%%%%%%%%%%%
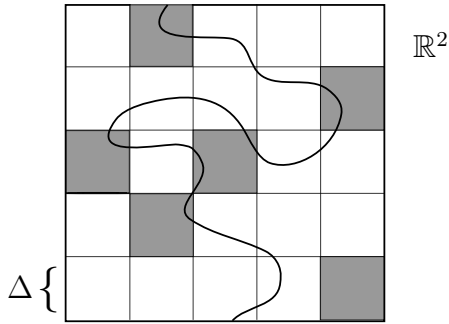
\begin{figure}[htb]
\vspace{0.2cm}
\begin{picture}(0.5,0.5)
\put(165,0){\includegraphics[height= 120pt,width=120pt]{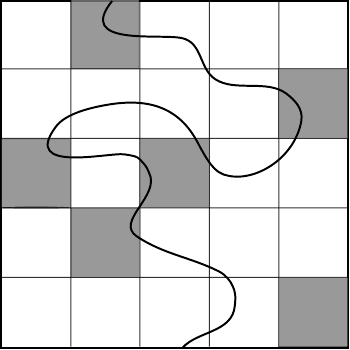}}
\end{picture}
\put(295,100){$\R^2$}
\put(143,9){$\Delta$}
\put(154,8){\Large\mbox{\{}}
%\put(360,-5){\small\mbox{space}}
%\put(223,-10){$\pi_0$}
%\put(300,127){$\pi_t$}
%\put(100,117){$t$}
%\put(100,33){$0$}
\caption{\small Block percolation in $\R^2$. Gray blocks are open. The curve represents a path in $\mathcal{W}_{\ell}$.} 
\label{fig:blockperc}
\end{figure}
%%%%%%%%%%%%%%%%%%%%%

For $w \in \mathcal{W}_\ell$, put
\begin{equation}\label{defpsi}
\psi(w):= \#\{k \in \Delta \Z^d \colon\, w \text{ intersects } B_{\Delta}(k) \text{ and } \Upsilon(k) = 1 \}
\end{equation}
and let
\begin{equation}\label{defPsi}
\Psi(\ell) := \sup_{w \in \mathcal{W}_\ell} \psi(w).
\end{equation}
In order to control $\Psi(\ell)$, we need to restrict the class of allowed percolative systems.
We will call a PS $\Upsilon$ \emph{homogeneous} with parameter $p \in (0,1)$ if $\Upsilon(k)$ has distribution Bernoulli($p$) for each $k \in \Delta \Z^d$. We call it (finitely) \emph{partitioned} if there exists a finite partition $\mathcal{P}$ of $\Delta \Z^d$ such that,
for each $I \in \mathcal{P}$,
\begin{equation}\label{indepb}
\left(\Upsilon(k)\right)_{k \in I} \text{ are jointly independent.}
\end{equation}
In other words, $\Upsilon$ is partitioned if its dependence graph has a finite chromatic number.
In that case, we let $|\mathcal{P}| := \# \mathcal{P}$. 
In the following, we use the abbreviation $p$-PPS for ``homogeneous partitioned percolative system with parameter $p$''.

\subsection{Key lemma}
\label{subsec:keylemma}

The following lemma is the key to the proof of Proposition~\ref{prop:PPSseq} below.

\begin{lemma}\label{lemma:tailPsi}
There exist constants $c_1, c_2 \in (0,\infty)$ depending on $d$  only such that,
for any percolative system $\Upsilon$ with scale $\Delta$ that is stochastically dominated by a 
$p$-PPS with partition $\mathcal{P}$,
\begin{equation}\label{eqtailPsi}
P\left( \Psi(\ell) > |\mathcal{P}| c_1 \frac{\theta \ell}{\Delta} \right) \le |\mathcal{P}| e^{- c_2 \left(\frac{\theta \ell}{\Delta}-1\right)}
\;\;\; \text{for any } \theta \in [p^{\frac{1}{d}},1].
\end{equation}
\end{lemma}
\noindent
Our proof of Lemma~\ref{lemma:tailPsi} is an adaptation of the proof of Lemma 8 in \cite{KeSi03}. 
It is based on geometric constraints of $\R^d$ and an application of Bernstein's inequality, 
which we recall for the case of i.i.d.\ bounded random variables.
\begin{lemma}\label{Bernstein} (Bernstein's inequality)
Suppose that $(X_i)_{i \in \N}$ is an i.i.d.\ sequence of a.s.\ bounded random variables with joint law $P$. Then
\begin{equation}\label{eqBernstein}
P\left(\sum_{i=1}^n X_i -EX_i > x\right) \le e^{-\frac{x}{2} \left(\|X_1\|_{\infty}+ \frac{n Var(X_1)}{x} \right)^{-1}}.
\end{equation}
\end{lemma}
For a proof of Lemma~\ref{Bernstein}, see e.g.\ Chow and Teicher \cite{ChTe88}, Exercise 4.3.14.
\begin{proof}[Proof of Lemma~\ref{lemma:tailPsi}]
There exist $K_1, K_2 \in \N$, depending on $d$ only, with the following properties. 
For any $\ell$ and $\Delta$, the total number of $\Delta$-blocks intersecting any path in $\mathcal{W}_\ell$
 is at most $K_1 \lceil \ell/\Delta \rceil$ and, 
 for any $n$ and $\Delta$, the number of connected subsets of $\R^d$ that are unions of 
exactly $n$ $\Delta$-blocks and contain the origin
is at most $e^{K_2n}$.
We will show that \eqref{eqtailPsi} holds with $c_2 := 2^d K_1 K_2$ and $c_1 := 16 c_2$.

Since $\Psi(\ell)$ does not decrease if additional $\Delta$-blocks are opened, 
we may suppose that $\Upsilon$ is a $p$-PPS with partition $\mathcal{P}$.
Let
\begin{equation}
L := \lceil \theta^{-1} \rceil, \;\; N := K_1 \lceil \ell/(L \Delta) \rceil.
\end{equation} 
As discussed in the first paragraph, $N$ is an upper bound 
for the number of $L\Delta$-blocks intersected by any path in $\mathcal{W}_\ell$. 
If $\ell/(L\Delta) < \frac12$, then $\theta \ell/\Delta < 1$
and \eqref{eqtailPsi} holds trivially. 
Therefore, we may assume that $\ell/(L\Delta) \ge \frac12$, 
in which case $N \le 3K_1 \ell/(L \Delta)$. Letting
\begin{equation}\label{ptailPsieq1}
\begin{array}{rl}
\mathscr{C}_L^N := \{ \hspace{-0.3cm} & \text{connected subsets of } \R^d \text{ containing the origin }\\
& \;\; \text{ that are the union of } N \text{ distinct } L\Delta \text{-blocks}\},
\end{array}
\end{equation}
we can estimate, for $x > 0$,
\begin{align}\label{ptailPsieq2}
P\left(\exists \; w \in \mathcal{W}_\ell \colon\, \psi(w) > x\right) 
& \le \sum_{C \in \mathscr{C}_L^N} P\left(\exists \; w \in \mathcal{W}_\ell, w \subset C \colon\, \psi(w) > x\right) \nonumber \\
& \le \sum_{C \in \mathscr{C}_L^N} P\left(\#\{\text{open } \Delta \text{-blocks in } C \} > x\right).
\end{align}
To estimate for a fixed $C \in \mathscr{C}_L^N$ the corresponding term in \eqref{ptailPsieq2}, we use the partition. 
\begin{align}\label{ptailPsieq3}
P\left(\#\{\text{open } \Delta \text{-blocks in } C \} > x\right) 
& \le \sum_{I \in \mathcal{P}}P\left(\#\{\text{open } \Delta \text{-blocks in } C\cap I\} > \frac{x}{|\mathcal{P}|}\right) \nonumber \\
& \le |\mathcal{P}|P\left(\textrm{Bin}(NL^d,p) > \frac{x}{|\mathcal{P}|} \right),
\end{align}
where $\textrm{Bin}(NL^d, p)$ is a Binomial random variable and \eqref{ptailPsieq3} 
is justified by \eqref{indepb} and the fact that each 
$C \in \mathscr{C}_L^N$ is the union of exactly $NL^d$ $\Delta$-blocks. 
By the definition of $L$ and our choice of $c_1$, we can check that $pNL^d < \frac12 c_1 \theta \ell/\Delta$. 
Therefore, substituting $x$ in \eqref{ptailPsieq3} by $|\mathcal{P}|c_1 \theta \ell / \Delta$ 
and applying Bernstein's inequality \eqref{eqBernstein}, we obtain
\begin{equation}\label{ptailPsieq4}
P\left(\#\{\text{open } \Delta \text{-blocks in } C \} > |\mathcal{P}| c_1 \frac{\theta \ell}{\Delta} \right) \le |\mathcal{P}| \exp\left(-\frac{c_1 \theta \ell}{8\Delta}\right).
\end{equation}
Since $N \le 3 K_1 \ell /(L\Delta) \le 3 K_1  \ell \theta /\Delta $,
we have $K_2 N < c_2 \theta \ell /\Delta$. Hence, combining \eqref{ptailPsieq2} and \eqref{ptailPsieq4}, we get
\begin{equation}\label{ptailPsieq5}
P\left( \Psi(\ell) > |\mathcal{P}| c_1 \frac{\theta \ell}{\Delta} \right) 
\le |\mathcal{P}| e^{ K_2 N - 2 c_2 \frac{\theta \ell}{\Delta}} \le |\mathcal{P}| e^{-c_2  \frac{\theta \ell}{\Delta}}.
\end{equation}
\end{proof}

\subsection{Sequences of percolative systems}
\label{subsec:seqPS}

The following proposition concerns sequences of percolative systems, and will be used in Section~\ref{sec:proofmainprop}
in the proof of Proposition~\ref{mainprop}.

\begin{proposition}\label{prop:PPSseq}
Let $(\Upsilon_r)_{r \in \N}$ be a sequence of percolative systems in $\R^d$ with with scales $\Delta_r$,
defined jointly in the same probability space through an arbitrary coupling.
Suppose that, for each $r \in \N$, $\Upsilon_r$ is stochastically dominated by a $p_r$-PPS
with partition $\mathcal{P}_r$ such that the following hold: 
\begin{equation}\label{PPSseqeq1}
\begin{array}{cl}
\text{(i) } & \limsup_{r \to \infty} |\mathcal{P}_r| < \infty.\\
\text{(ii) } & m := \limsup_{r \to \infty} r^{-1} \log(\Delta_r)  < \infty.\\
\text{(iii) } & M:= - \limsup_{r \to \infty} r^{-1} \log(p_r)  > md.
\end{array}
\end{equation}
Then, for any $\kappa \in (0,(md)^{-1})$,
\begin{equation}\label{PPSseqeq2}
\lim_{n \to \infty} \limsup_{\ell \to \infty} \frac{1}{\ell} 
\sum_{r = n}^{\lfloor \kappa \log(\ell) \rfloor} \Delta^d_r \Psi_r(\ell) = 0 \quad \text{a.s.}
\end{equation}
where $\Psi_r(\ell)$ is defined for $\Upsilon_r$ as in \eqref{defpsi}--\eqref{defPsi}.
\end{proposition}
\begin{proof}
Let $ 0 < \varepsilon < \frac12(1/\kappa -md)$ and 
put $\theta_r := \sqrt[d]{p_r} \vee e^{-b r}$ with $b = m(d-1) + \varepsilon$.
By (i), there exists $K_1 \in (0,\infty)$ such that $\sup_r |\mathcal{P}_r| \le K_1$
and, by (ii), there exist $K_2 \in (0,\infty)$ and $r_0 \in \N$ such that $\Delta_r \le K_2 e^{(m+\varepsilon) r}$ whenever $r \ge r_0$. 
Hence, by Lemma~\ref{lemma:tailPsi},
\begin{align}\label{pPPSseqeq1}
P\left(\exists \; r_0 \le r \le \lfloor \kappa \log(\ell) \rfloor \colon\, \Psi_r(\ell) > K_1 c_1 \ell \frac{\theta_r}{\Delta_r}\right)
& \le K_1 \kappa \log(\ell) e^{c_2} \exp\left(-c_2\ell\frac{\ell^{-\kappa b}}{K_2 \ell^{\kappa (m+ \varepsilon)}}\right) \nonumber \\
& = K_1 \kappa \log(\ell) e^{c_2} \exp\left(-\frac{c_2\ell^a}{K_2}\right),
\end{align}
where $a := 1 - \kappa(m+\varepsilon+b) >0$ by our choice of $\varepsilon$ and $b$.
Thus, \eqref{pPPSseqeq1} is summable in $\ell$. By the Borel-Cantelli lemma, 
a.s.\ for $n \ge r_0$ and $\ell$ large enough we may estimate
\begin{equation}\label{pPPSseqeq2}
\frac{1}{\ell}\sum_{r = n}^{\lfloor \kappa \log(\ell) \rfloor} \Delta^d_r \Psi_r(\ell) 
\le K_1 c_1 \sum_{r = n}^{\lfloor \kappa \log(\ell) \rfloor}\Delta^{d-1}_r\theta_r.
\end{equation}
By (ii)-(iii) and the definition of $\theta_r$, $\Delta^{d-1}_r\theta_r$ is summable in $r$. 
Therefore \eqref{PPSseqeq2} follows by first letting $\ell \to \infty$ and then $n \to \infty$.
\end{proof}

\section{Proof of Proposition~\ref{mainprop}}
\label{sec:proofmainprop}

Section~\ref{subsec:proofmainpropa} contains the proof of Proposition~\ref{mainprop}(a),
Section~\ref{subsec:proofmainpropb} the proof of Proposition~\ref{mainprop}(b).

Most of the work is concentrated in Section~\ref{subsec:proofmainpropa},
where the renormalization scheme for rarefied blocks is defined
and analyzed using the results from Section~\ref{sec:PPS}.
Central to this work are estimates for systems of independent simple random walks,
stated in Lemma~\ref{lemma:ISRW} below, which are used for comparison
with the system of \emph{holes} of the SSEP via a result due to Liggett.
These estimates are used to control a \emph{recursive formula} that, roughly speaking,
transfers properties from larger to smaller scales, allowing us to deduce
microscopic properties from mesoscopic and macroscopic properties.

In Section~\ref{subsec:proofmainpropb}, a similar approach is used
to analyze turbulent blocks from the point of view of Section~\ref{sec:PPS}.
There the construction and estimates are much simpler.

\subsection{Proof of Proposition~\ref{mainprop}(a)}
\label{subsec:proofmainpropa}

\subsubsection{Bad blocks}
\label{subsubsec:badblocks}

Fix $\rho_- \in (0,\rho)$, let $N_0 \in \N$ be large enough such that
\begin{equation}\label{defrhoinfty}
\bar{\rho}_{\infty} := \prod_{r=1}^{\infty}(1-N_0^{-r/4}) \ge 1-\rho_- =: \bar{\rho}_+
\end{equation}
and put
\begin{equation}\label{defbarrhor}
\bar{\rho}_r:=\prod_{k=1}^r(1-N_0^{-k/4}).
\end{equation}
Set
\begin{equation}\label{defomegaDeltarhor}
\omega_r := N_0^r, \;\; \Delta_r := N_0^{6r} \; \text{ and } \; \rho_r:=1-\bar{\rho}_r.
\end{equation}
The parameters $\epsilon_r$ will be defined in Section~\ref{subsec:proofmainpropb}.
Set also $\bar{\rho} := 1-\rho$ and, for $\eta \in \{0,1\}^{\Z}$, define $\bar{\eta}$ by
\begin{equation}\label{defbareta}
\bar{\eta}(x) := 1 - \eta(x).
\end{equation}

In the following, we will also need $r$-superblocks, defined as
\begin{equation}\label{defrsupblocks} 
\mathbf{B}_r(k,s) := [k-5\Delta_r,k+6\Delta_r)\times[s-2\Delta_r,s+\Delta_r), \quad k,s \in \Delta_r\Z.
\end{equation}

We call the $r$-block $B_r(k,s)$ \emph{bad} if $\mathbf{B}_r(k,s)$ is $r$-rarefied.
Thus, any $r$-rarefied $r$-block is bad. 
We call $r$-\emph{dense} any set in $\R^2$ that is not $r$-rarefied.

\begin{lemma}\label{lemma:nobigraref}
For any $\kappa >0$, $\mathbb{P}_{\nu_\rho}$-a.s.\ there exists 
a (random) $\ell_0 \in (0,\infty)$ such that,
if $\ell\ge\ell_0$, no bad $r$-blocks with 
$r \ge \lfloor \kappa \log(\ell) \rfloor $ intersect $[-\ell,\ell]^2$.
\end{lemma}
\begin{proof}
Since the product Bernoulli measure $\nu_\rho$ is a translation-invariant equilibrium, 
for any $r \in \N$, $x \in \Z$ and $t \in \R$, we have
\begin{align}\label{nobigrarefeq1}
\mathbb{P}_{\nu_\rho}\left(\Sigma_r^x(\xi_t)<\rho_r \omega_r\right) 
& \le \mathbb{P}_{\nu_\rho}\left(\Sigma_r^x(\xi_t)<\rho_- \omega_r \right) \nonumber \\
& = P\left(\textrm{Bin}(\omega_r,\rho) - \rho \omega_r < -(\rho-\rho_-)\omega_r\right) \le e^{-\varepsilon \omega_r},
\end{align}
where $\textrm{Bin}(\omega_r,\rho)$ is a Binomial random variable and $\varepsilon > 0$. 
The last step can be justified e.g.\ by using Bernstein's inequality \eqref{Bernstein}.
Therefore, for any $(k,s) \in \Delta_r \Z^2$,
\begin{align}\label{nobigrarefeq2}
\mathbb{P}_{\nu_\rho}\left(B_r(k,s) \text{ is bad} \right) 
& \le \sum_{(x,t)\in \mathbf{B}_r(k,s)\cap \Z^2}\mathbb{P}_{\nu_\rho} \left( \Sigma_r^x(\xi_t) < \rho_r \omega_r\right) \nonumber \\
& \le 33 \Delta_r^2 e^{-\varepsilon \omega_r} \le C e^{-\frac{\varepsilon}{2}N_0^r}
\end{align}
for some $C \in (0,\infty)$.
Since at most $(2\ell + 1)^2$ $r$-blocks intersect $[-\ell,\ell]^2$, we can estimate
\begin{align}\label{nobigrarefeq3}
& \mathbb{P}_{\nu_\rho} \left( \exists \; r > \kappa \log(\ell) \text{ and a bad } r\text{-block intersecting } [-\ell,\ell]^2 \right) \nonumber \\
& \quad \;\; \le C(2\ell+1)^2 \sum_{r=\lfloor k \log(\ell) \rfloor}^{\infty} e^{-\frac{\varepsilon}{2}N_0^r}
\end{align}
which is summable in $\ell$, and so the claim follows by the Borel-Cantelli lemma.
\end{proof}

\subsubsection{Locally spoiled blocks}
\label{subsubsec:locspoilblocks}

For $(k,s) \in \Delta_r\Z^2$, let
\begin{equation}\label{defneighb}
\mathcal{B}_r(k,s) := [k-\Delta_r, k+2\Delta_r) \times [s-\Delta_r,s+\Delta_r)
\end{equation}
be the \emph{neighbourhood} of the $r$-block $B_r(k,s)$, and let
\begin{equation}\label{defbase}
\mathbf{\Lambda}_r(k,s) := [k-5\Delta_r, k+6\Delta_r) \times \{s-2\Delta_r\}
\end{equation}
be the \emph{base} of the $r$-superblock $\mathbf{B}_r(k,s)$. 
Define also $\mathbf{V}_r^k := [k-5\Delta_r, k+6\Delta_r) \subset \Z$, so that
$\mathbf{\Lambda}_r(k,s) = \mathbf{V}_r^k \times \{s - 2\Delta_r\}$.
See Figure~\ref{fig:relposblocks}.
We also need the \emph{interior} of $\mathbf{B}_r(k,s)$,
\begin{equation}\label{defintsblock}
\mathring{\mathbf{B}}_r(k,s) := [k-5\Delta_r+1,k+6\Delta_r -1) \times [s - 2\Delta_r, s + \Delta_r),
\end{equation}
and, for $(x,t) \in \mathcal{B}_{r+1}(k,s)$,
\begin{align}\label{defsigmahat}
\widehat{\Sigma}_r^{k,s}(x,t) := \#\{ & \text{all particles of the SSEP in } [x,x+ \omega_r)\times \{t\} \text{ that  } \nonumber \\
& \text{ stayed in } \mathring{\mathbf{B}}_{r+1}(k,s) \text{ during the time interval } [s-2\Delta_{r+1}, t]\}.
\end{align}
%Clearly, $\widehat{\Sigma}_r^{k,s}(x,t) \le \Sigma_r^x(\xi_t)$.

%%%%%%%%%%%%%%%%% FIGURE RELATIVE POSITION OF BLOCKS %%%%%%%%%%%%%%%%%%%%%%%%%%%%
\begin{figure}[hbtp]
\begin{center}
\setlength{\unitlength}{0.5cm}
\vspace{-0.5cm}
\begin{picture}(20,10)(0,0)
%horizontal lines
{\linethickness{0.007cm}
\put(0,1){\line(1,0){22}} 
\put(0,3){\line(1,0){22}}
\put(0,5){\line(1,0){22}}
\put(0,7){\line(1,0){22}}
}
%vertical lines
{\linethickness{0.007cm}
\put(0,1){\line(0,1){6}}\put(2,1){\line(0,1){6}}
\put(4,1){\line(0,1){6}} \put(6,1){\line(0,1){6}}
\put(8,1){\line(0,1){6}} \put(10,1){\line(0,1){6}}
\put(12,1){\line(0,1){6}} \put(14,1){\line(0,1){6}}
\put(16,1){\line(0,1){6}} \put(18,1){\line(0,1){6}}
\put(20,1){\line(0,1){6}} \put(22,1){\line(0,1){6}}
}
%thicklines
{\linethickness{0.05cm}
\put(0,1){\line(1,0){22}}
\put(0,7){\line(1,0){22}}
\put(0,1){\line(0,1){6}}
\put(22,1){\line(0,1){6}}
\put(8,3){\line(1,0){6}}
\put(8,7){\line(1,0){6}}
\put(8,3){\line(0,1){4}}
\put(14,3){\line(0,1){4}}
\put(10,5){\line(1,0){2}}
\put(10,5){\line(0,1){2}}
\put(12,5){\line(0,1){2}}
}
%arrows
\put(11,7.5){\vector(0,-1){1.7}}
\put(5,-0.5){\vector(1,0){10}}
\put(-2.7,2){\vector(0,1){3.5}}
\put(23.7,1){\vector(-1,0){1.4}}
%symbols
\put(-0.9,4.8){$s$}
\put(9.8,0){$k$}
\put(-2,1.7){$\Delta_r$}\put(-0.9,1.6){\huge\mbox{\{}}
\put(9.6,8){$B_r(k,s)$}
\put(8.97,3.7){$\mathcal{B}_r(k,s)$}
\put(16.97,5.5){$\mathbf{B}_r(k,s)$}
\put(22.9,1.3){$\Lambda_r(k,s)$}
\put(-3.4,6){\small\mbox{time}}
\put(15.5,-0.7){\small\mbox{space}}
\end{picture}
\end{center}
\caption{\small Relative position of $B_r(k,s)$, $\mathcal{B}_r(k,s)$, $\mathbf{B}_r(k,s)$ and $\Lambda_r(k,s)$.} \label{fig:relposblocks}
\end{figure}
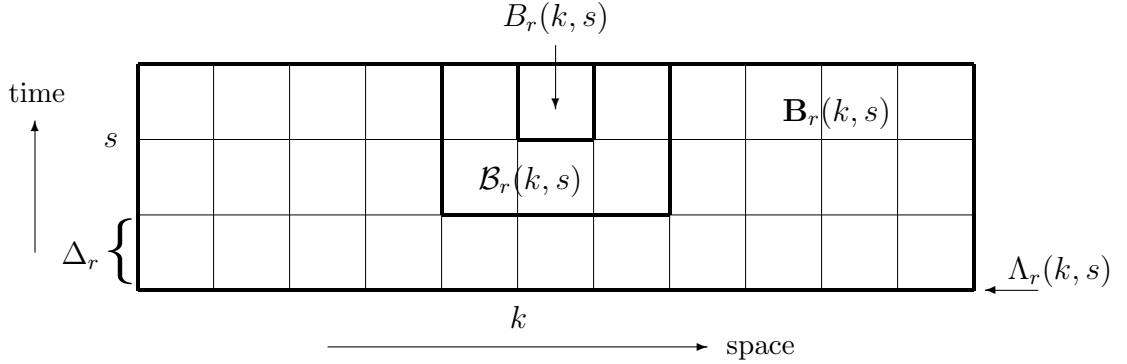
%%%%%%%%%%%%%%%%%%%%%%%%%%%%%%%%%%%%%%%%%%%%%%%%%%%%%%%%%%%%%%%%%%%%%%%%%%%%%%

A block $B_{r+1}(k,s)$ is called \emph{locally spoiled} 
if $\mathbf{\Lambda}_{r+1}(k,s)$ is $(r+1)$-dense 
but there is a point $(x,t)$ such that $[x,x+\omega_r)\times \{t\} \subset \mathcal{B}_r(k,s)$ 
and $\widehat{\Sigma}_r^{k,s}(x,t) < \rho_r \omega_r$.
Being locally spoiled means that, in the scale $\Delta_{r+1}$,
the $(r+1)$-block ``has good conditions'', meaning that the base of its $(r+1)$-superblock is $(r+1)$-dense, 
but nonetheless there are not enough particles transfered locally (i.e., inside $\mathbf{B}_{r+1}(k,s)$) to 
ensure that in the finer scale $\Delta_r$ the neighbourhood $\mathcal{B}_{r+1}(k,s)$ is $r$-dense 
(which would in turn guarantee that $B_{r+1}(k,s)$ contains no bad $r$-blocks).
We will see below that, with our choice of parameters, being locally spoiled is an extremely unlikely event.

%A block $B_{r+1}(k,s)$ will be called \emph{spoiled} if $\mathbf{\Lambda}_{r+1}(k,s)$ is not $(r+1)$-compressed 
%but $B_{r+1}(k,s)$ contains a bad $r$-block.
%Note that the latter does not imply that $B_{r+1}(k,s)$ itself is $r$-compressed,
%but it does imply that its neighbourhood $\mathcal{B}_{r+1}(k,s)$ is $r$-compressed.
%Being spoiled means that, in the scale $\Delta_{r+1}$,
%the $(r+1)$-block ``has good conditions'', meaning that the base of its $(r+1)$-superblock is
%not $(r+1)$-compressed, but nonetheless in the finer scale $\Delta_r$ it contains bad $r$-blocks. 
%We will see below that, with our choice of parameters, being spoiled is an extremely unlikely event.

Define a percolative system $\Upsilon_r$ with scale $\Delta_r$ by
\begin{equation}\label{defUpsilonr}
\Upsilon_r(k,s) := \mathbbm{1}_{\{B_r(k,s) \text{ is locally spoiled}\}},
\end{equation}
and, for each $a = (a_1,a_2) \in \mathbf{B}_r(0,2\Delta_r) \cap \Delta_r \Z^2$, 
let
\begin{equation}\label{defclassa}
I_a := \{(z_1,z_2)\in \Delta\Z^2 \colon\, z_1 \equiv a_1 \hspace{-0.3cm} \pmod{11} \; \text{ and } \; z_2 \equiv a_2 \hspace{-0.3cm} \pmod{3}\}.
\end{equation}
Then
\begin{equation}\label{defcurlyPr}
\mathcal{P}_r :=\{I_a \, \colon\, a \in \mathbf{B}_r(0,2 \Delta_r) \cap \Delta_r \Z^2\}
\end{equation} 
is a partition of $\Delta_r\Z^2$ with $|\mathcal{P}_r| = 33$.

\begin{lemma}\label{lemma:stochdoma}
For all large enough $r \in \N$, $\Upsilon_r$ is stochastically dominated by a $p_r$-PPS with
partition $\mathcal{P}_r$, where $p_r$ tends to $0$ super-exponentially fast as $r \to \infty$.
\end{lemma}

The proof of Lemma~\ref{lemma:stochdoma} requires quite a bit of work, 
including estimates for systems of simple random walks for comparison with the SSEP. 
Therefore, we postpone it to Section~\ref{subsubsec:proofmainlemmaa}, and show first how 
it is used to prove Proposition~\ref{mainprop}(a).

\subsubsection{Proof of Proposition~\ref{mainprop}(a)}
\label{subsubsec:proofmainpropa}

\begin{proof}
Let
\begin{equation}\label{pmainpropaeq1}
\begin{array}{rcl}
\Phi_r^{\mathrm{b}}(\ell) & := & \sup_{w \in \mathcal{W}_\ell} \#\{\text{bad } r\text{-blocks that intersect } w\},\\
\Psi_r^{\mathrm{ls}}(\ell) & := & \sup_{w \in \mathcal{W}_\ell} \#\{\text{locally spoiled } r\text{-blocks that intersect } w\}.
\end{array}
\end{equation}
Since $\Phi_r^\mathrm{r}(\ell) \le \Phi_r^\mathrm{b}(\ell)$, 
it is enough to prove that
\begin{equation}\label{pmainpropaeq2}
\lim_{r \to\infty}\limsup_{\ell \to \infty}\ell^{-1} \Delta^2_r \Phi_r^\mathrm{b}(\ell)=0.
\end{equation}
We claim that, for all $r \in \N$,
\begin{equation}\label{pmainpropaeq3}
\Phi_r^\mathrm{b}(\ell) \le N_0^{12}\Phi_{r+1}^{\mathrm{b}}(\ell) + N_0^{12}\Psi_{r+1}^{\mathrm{ls}}(\ell).
\end{equation}
Indeed, if an $r$-block is bad, then the unique $(r+1)$-block containing it is either bad or locally spoiled, 
and the number of $r$-blocks inside any given $(r+1)$-block is equal to $N_0^{12}$.
By induction we get, for $R \ge r+1$,
\begin{equation}\label{pmainpropaeq4}
\Delta^2_r\Phi_r^{\mathrm{b}}(\ell) \le \Delta^2_R \Phi_R^\mathrm{b}(\ell) + \sum_{n=r+1}^{R} \Delta^2_n \Psi_n^\mathrm{ls}(\ell).
\end{equation}
For $\kappa \in (0,(12\log(N_0))^{-1})$, take $\ell_0$ as in Lemma~\ref{lemma:nobigraref}
and $R=\lfloor \kappa \log(\ell) \rfloor$. 
Then, for $\ell \ge \ell_0$, we may estimate
\begin{equation}\label{pmainpropaeq5}
\frac{1}{\ell}\Delta^2_r\Phi_r^\mathrm{b}(\ell) \le \frac{1}{\ell} \sum_{n=r+1}^{\lfloor \kappa \log(\ell) \rfloor} \Delta^2_n \Psi_n^\mathrm{ls}(\ell),
\end{equation}
and so \eqref{pmainpropaeq2} follows from Lemma~\ref{lemma:stochdoma} and Proposition~\ref{prop:PPSseq}.
\end{proof}

The rest of this section is dedicated to the proof of Lemma~\ref{lemma:stochdoma}.
In Section~\ref{subsubsec:ISRW} we derive some estimates for systems of independent simple random walks.
These are used in Lemma~\ref{lemma:ISRWtoSSEP} below for comparison with the system of holes of the SSEP.
The latter lemma is used in Section~\ref{subsubsec:proofmainlemmaa} to prove Lemma~\ref{lemma:stochdoma}.

\subsubsection{Estimates for systems of independent random walks}
\label{subsubsec:ISRW}

It will be useful to compare the system $\bar{\xi}$ of the \emph{holes} of the exclusion process
with a system of independent simple random walks, which we define next.

Let $(S^{z})_{z \in \Z}$ be a collection of independent simple random walks on $\Z$, with $S^z_0 = z$ for each $z \in \Z$.
For $\eta \in \{0,1\}^{\Z}$, define the process $\xi^{\circ} = (\xi^{\circ}_t)_{t \ge 0}$ by
\begin{equation}\label{defISRW}
\xi^{\circ}_t(x) := \sum_{z\in \Z} \eta(z) \mathbbm{1}_{\{S^z_t = x\}}, \quad (x,t) \in \Z \times [0,\infty).
\end{equation}
The interpretation is that, if we launch from each site $z$ with $\eta(z)=1$ an independent simple random walk,
then $\xi^{\circ}_t(x)$ is the number of random walks present at the site $x$ at time $t$.

The following lemma states two estimates for $\Sigma_r^x(\xi^{\circ}_t)$, where $[x, x+ \omega_r)\times\{t\} \subset \mathcal{B}_{r+1}(0, 2\Delta_{r+1})$.
The first gives a bound on its exponential moments in terms of its first moment, 
while the second gives a bound on the first moment in terms of density properties of the initial configuration in the $(r+1)$-scale.

\begin{lemma}\label{lemma:ISRW}
Let $\eta \in \{0,1\}^{\Z}$ and $\xi^{\circ} = (\xi^{\circ}_t)_{t \ge 0}$ be a system of independent SRWs, as discussed above,
starting from $\bar{\eta}$ (recall \eqref{defbareta}).
Then the following hold:
\begin{enumerate}
\item[(i)] For any $\lambda > 0$, $x \in \Z$ and $t\ge0$,
\begin{equation}\label{ISRWeq1}
\mathbb{E}_{\bar{\eta}}\left[\exp({\lambda \Sigma_r^x(\xi^{\circ}_t)}) \right] 
\le \exp\left\{{(e^\lambda-1) \mathbb{E}_{\bar{\eta}} \left[\Sigma_r^x(\xi^{\circ}_t)\right]}\right\}.
\end{equation}
\item[(ii)] For large enough $r \in \N$ and any $(x,t) \in \mathcal{B}_{r+1}(0,2\Delta_{r+1})$, 
\begin{equation}\label{ISRWeq2}
\mathbb{E}_{\bar{\eta}} \left[\Sigma_r^x(\xi^{\circ}_t)\right] \le 1+\bar{\rho}_{r+1} \omega_r \; \text{ if } \mathbf{\Lambda}_{r+1}(0,2\Delta_{r+1}) \text{ is } (r+1)\text{-dense},
\end{equation}
i.e.,
$\sum_{y \in [x,x+\omega_{r+1})} \eta(y) \ge \rho_{r+1} \omega_{r+1}$ for any $x \in \Z$ such that 
$[x,x+\omega_{r+1}) \times \{0\} \subset \mathbf{\Lambda}_{r+1}(0,2\Delta_{r+1})$.
\end{enumerate}
\end{lemma}
\begin{proof}
\textit{(i)} Using \eqref{defISRW}, we may write
\begin{align}\label{pISRWeq1}
\mathbb{E}_{\bar{\eta}}\left[\exp(\lambda \Sigma_r^x(\xi^{\circ}_t))\right]
& = \prod_{z \in \Z} E\left[e^{\lambda \bar{\eta}(z)\mathbbm{1}_{\{S^z_t \in [x,x+\omega_r)\}}} \right]\nonumber\\
& = \prod_{z \in \Z} \left\{ \bar{\eta}(z) (e^{\lambda}-1)  P(S^z_t \in [x,x+\omega_r)) + 1 \right\} \nonumber\\
& \le \prod_{z \in \Z} \exp\left(\bar{\eta}(z) (e^{\lambda}-1)  P(S^z_t \in [x,x+\omega_r))\right) \nonumber\\
& = \exp\left\{(e^{\lambda}-1)\mathbb{E}_{\bar{\eta}}\left[\Sigma_r^x(\xi^{\circ}_t)\right] \right\}.
\end{align}

\textit{(ii)} We recall two basic results for one-dimensional simple random walk: 
there exist $K_1,K_2 \in (0,\infty)$ such that
\begin{equation}\label{SSRW1}
P\left(|S_t^0| > 2\sqrt{t}\log t \right) \le K_1 e^{-K_2 (\log t)^2}, \;\; t \ge 1,
\end{equation}
and
\begin{equation}\label{SSRW2}
|P\left(S^y_t = z_1\right)-P\left(S^y_t = z_2\right)| \le K_1\frac{|z_1-z_2|}{t}, \;\; y,z_1,z_2 \in \Z, \; t \ge 1.
\end{equation}
The first of these can be verified e.g.\ with the help of Bernstein's inequality \eqref{Bernstein};
for the second, see e.g.\ Lawler and Limic \cite{LaLi10}, Theorem 2.3.5.

To simplify the notation, in the following we omit the coordinates $(0,2\Delta_{r+1})$ of the sets involved. 
Let $k_t := \lceil 2\sqrt{t}\log(t)/\omega_{r+1}\rceil$ 
and put $A_t^x := [x-k_t \omega_{r+1},x+(k_t+1)\omega_{r+1})$. 
Since $(x,t) \in \mathcal{B}_{r+1}$, we have $A_t^x \times \{0\} \subset \mathbf{\Lambda}_{r+1}$. 
Write
\begin{align}\label{pISRWeq2}
\mathbb{E}_{\bar{\eta}}\left[\Sigma_r^x(\xi^{\circ}_t)\right]
& =\sum_{z \in \Z} \bar{\eta}(z)P\left(S^z_t \in [x,x+\omega_r)\right) \nonumber\\
& \le \sum_{z \notin A_t^x} P\left(S^z_t \in [x,x+\omega_r)\right) +
\sum_{z \in A_t^x} \bar{\eta}(z)P\left(S^z_t \in [x,x+\omega_r)\right).
\end{align}
The first term in the r.h.s.\ of \eqref{pISRWeq2} can be estimated by
\begin{align}\label{pISRWeq3}
\sum_{y \in [x,x+\omega_r)}P\left(S^y_t \notin A_t^x\right)
\le \omega_r P\left(|S_t| > 2\sqrt{t}\log t\right) \le K_1\omega_r e^{-K_2 (\log \Delta_{r+1})^2} \le \frac{1}{2}
\end{align}
for $r$ large enough, where we use \eqref{SSRW1} and the fact that $t \ge \Delta_{r+1}$.
Decompose $A_t^x$ into disjoint intervals $I_1,\ldots,I_n$ with length exactly $\omega_{r+1}$, 
and let $z_i \in I_i$ be the maximizer of $z \mapsto P(S^z_t \in [x,x+\omega_r))$ in $I_i$. 
Then the second term in \eqref{pISRWeq2} is at most
\begin{align}\label{pISRWeq4}
\sum_{i=1}^n\sum_{z \in I_i} \bar{\eta}(z)P\left(S^{z_i}_t \in [x,x+\omega_r)\right)
& \le \bar{\rho}_{r+1} \omega_{r+1} \sum_{i=1}^n P\left(S^{z_i}_t \in [x,x+\omega_r)\right) \nonumber \\
& = \bar{\rho}_{r+1} \sum_{i=1}^n\sum_{z \in I_i} P\left(S^{z_i}_t \in [x,x+\omega_r)\right).
\end{align}
The last double sum in the r.h.s.\ of \eqref{pISRWeq4} is bounded by
\begin{equation}\label{pISRWeq5}
\sum_{z \in A_t^x} P \left(S^{z}_t \in [x,x+\omega_r)\right) 
\; + \sum_{y \in [x,x+\omega_r)}\sum_{i=1}^n\sum_{z \in I_i} |P\left(S^{z_i}_t =y\right)-P\left(S^z_t =y\right)|.
\end{equation}
The first term in \eqref{pISRWeq5} can be estimated by
\begin{equation}\label{pISRWeq6}
\sum_{y \in [x,x+\omega_r)}P \left(S^{y}_t \in A_t^x\right) \le \omega_r,
\end{equation}
and, via \eqref{SSRW2}, the second term in \eqref{pISRWeq5} by
\begin{align}\label{pISRWeq7}
\omega_r |A_t^x| K_1\frac{\omega_{r+1}}{t} 
& \le 4 K_1 \frac{\omega_r \omega_{r+1} \log(t)}{\sqrt{t}} + 3K_1 \frac{\omega_r\omega_{r+1}^2}{t} \nonumber\\
& \le \frac{4K_1}{N_0^{r-1}}\left\{ \log(3N_0^{6(r+1)}) + \frac{1}{N_0^{2r-1}} \right\} \le \frac{1}{2}
\end{align}
for large enough $r$, where for the second inequality we use that $\Delta_{r+1} \le t\le 3\Delta_{r+1}$.
Now \eqref{ISRWeq2} follows by combining \eqref{pISRWeq2}--\eqref{pISRWeq7} since $\bar{\rho}_{r+1} \le 1$.
\end{proof}

\subsubsection{Proof of Lemma~\ref{lemma:stochdoma}}
\label{subsubsec:proofmainlemmaa}

In this section we give the proof of Lemma~\ref{lemma:stochdoma}.
The first step is to compare $\bar \xi$ with a system of independent simple random walks
and use the estimates of Lemma~\ref{lemma:ISRW} to show that, if $\Lambda_{r+1}(k,s)$ is $(r+1)$-dense,
then it is extremely unlikely for $\mathcal{B}_{r+1}(k,s)$ to be $r$-rarefied. 
This will also imply that the probability to have a locally spoiled $B_{r+1}(k,s)$ is extremely low, 
since particles in the SSEP, with large probability, do not travel very large distances in a short time.
This is the content of Lemma~\ref{lemma:ISRWtoSSEP} below.

We will need the following $\sigma$-algebras:
\begin{equation}\label{defsigals}
\mathcal{F}_r^{s} := \sigma \big( \xi_t \colon\, t \in (-\infty,s-2\Delta_r]\big), \;\; r \in \N, \; s \in \Delta_r \Z.\\
\end{equation}

\begin{lemma}\label{lemma:ISRWtoSSEP}
There exist $C_1,C_2 \in (0,\infty)$ such that, 
for all $r\in\N$ large enough, $(k,s) \in \Delta_{r+1}\Z^2$ and $(x,t) \in \mathcal{B}_{r+1}(k,s) \cap \Z^2$,
if $\mathbf{\Lambda}_{r+1}(k,s)$ is $(r+1)$-dense, then
\begin{equation}\label{ISRWtoSSEPeq}
\mathbb{P}_{\nu_\rho}\left( \widehat{\Sigma}_r^{k,s}(s,t) < \rho_r \omega_r \mid \mathcal{F}_{r+1}^{s}\right)
\le C_1 e^{-C_2 \sqrt{\omega_r}}.
\end{equation}
\end{lemma}
\begin{proof}
By translation invariance and the Markov property, it is enough
to prove \eqref{ISRWtoSSEPeq} for $(k,s)=(0,2\Delta_{r+1})$ and under $\mathbb{P}_{\eta}$ for an 
arbitrary $\eta \in \{0,1\}^{\Z}$, under the assumption that $\mathbf{\Lambda}_{r+1}(0, 2\Delta_{r+1})$ 
is $(r+1)$-dense. We will first do this for $\Sigma_r^x(\xi_t)$. 

We claim that, for any $\eta \in \{0,1\}^\Z$ and $\lambda > 0$,
\begin{equation}\label{ISRWtoSSEPeq1}
\mathbb{E}_{\eta}\left[\exp({\lambda \Sigma_r^x(\xi_t)}) \right] 
\le \mathbb{E}_{\eta}\left[\exp({\lambda \Sigma_r^x(\xi^{\circ}_t)}) \right]
\end{equation}
where $\xi^\circ$ is a system of independent simple random walks as in Lemma~\ref{lemma:ISRW}.
This can be justified using a result due to Liggett \cite{Li85}, Chapter VIII, Proposition 1.7,
by noting that, for any $n \in \N$, the function 
$(y_1,\ldots,y_n) \mapsto \exp\lambda \sum_{i=1}^n \mathbbm{1}_{[x,x+\omega_r)}(y_i)$
is symmetric and positive definite. 
Liggett's result only applies to initial configurations with finitely many particles,
but, since $\Sigma_r^x(\xi_t)$ is monotone in $\eta$, 
\eqref{ISRWtoSSEPeq1} follows by the monotone convergence theorem.

Since $\bar{\xi}$ under $\mathbb{P}_{\eta}$ has the same distribution as $\xi$ under $\mathbb{P}_{\bar{\eta}}$,
we have, by Markov's inequality, \eqref{ISRWtoSSEPeq1} and Lemma~\ref{lemma:ISRW}, that, 
for any $\lambda >0$ and $r$ large enough,
\begin{align}\label{ISRWtoSSEPeq2}
\mathbb{P}_{\eta}\left(\Sigma_r^x(\xi_t) < \rho_r \omega_r \right) 
& = \mathbb{P}_{\bar \eta}\left(\Sigma_r^x(\xi_t) > \bar{\rho}_r \omega_r \right) \\
& \le \exp \left\{ \left(e^{\lambda}-1\right)(1 + \bar{\rho}_{r+1}\omega_r ) - \lambda \bar{\rho}_r \omega_r \right\} \nonumber\\
& = e^{e^\lambda-1}\exp \bar{\rho}_r \omega_r \left\{ \left(e^{\lambda}-1\right) \frac{\bar{\rho}_{r+1}}{\bar{\rho}_r} - \lambda \right\}.
\end{align}
Using $e^{\lambda}-1 \le \lambda e^\lambda$ and the definition of $\bar{\rho}_r$, 
we see that the term in brackets in the r.h.s.\ of \eqref{ISRWtoSSEPeq2} 
is at most $\lambda e^\lambda (\lambda-\omega_r^{-1/4})$. 
Choosing $\lambda = \frac12 \omega_r^{-1/4}$, we obtain
\begin{equation}\label{ISRWtoSSEPeq3}
\mathbb{P}_{\eta}\left(\Sigma_r^x(\xi_t) < \rho_r \omega_r \right) 
 \le e^{\sqrt{e}-1} \exp - \, \frac{ \bar{\rho}_+ \sqrt{e}\sqrt{\omega_r}}{4} 
 = \tilde{C}_1 e^{ -\tilde{C}_2\sqrt{\omega_r}}.
\end{equation}
To obtain \eqref{ISRWtoSSEPeq3} for $\widehat{\Sigma}_r(x,t)$ 
in place of $\Sigma_r^x(\xi_t)$ (with possibly different constants),
we will argue that the two are equal with a uniformly large probability.

To that effect, let $(y^-_t)_{t \ge 0}$ denote the path starting at time $0$ 
from the point $y^-_0:=-5\Delta_{r+1}$ that goes upwards in time and (forcibly) 
jumps across any arrows of the graphical representation to the right.
Likewise, let $(y^+_t)_{t \ge 0}$ denote the path that starts at $y^+_0 := 6\Delta_{r+1}-1$ 
and follows the arrows of the graphical representation to the left. 
We see that, on the event $A:=\{y^-,y^+ \text{ do not hit } \mathcal{B}_{r+1}\}$,
no particles can move from outside $\mathbf{B}_{r+1}$ into $\mathcal{B}_{r+1}$.
In particular, $\widehat{\Sigma}_r(x,t) = \Sigma_r^x(\xi_t)$ on $A$ 
if $[x,x+\omega_r)\times \{t\} \subset \mathcal{B}_{r+1}$.
On the other hand, $y^-_t - y^-_0$ and $y^+_0-y^+_t$ are both distributed 
as a rate $1$ Poisson process, and are independent of $\xi_0$;
therefore, because of the shape chosen for $\mathbf{B}_{r+1}$, we have
\begin{equation}\label{ISRWtoSSEPeq4}
\mathbb{P}_{\eta}(A) \ge 1 - C e^{-\varepsilon\Delta_{r+1}}
\end{equation}
for some $C, \varepsilon \in (0,\infty)$, which completes the proof.
\end{proof}

\begin{proof}[Proof of Lemma~\ref{lemma:stochdoma}]

If $\mathbf{\Lambda}_{r+1}(k,s)$ is $(r+1)$-dense, then
we may use \eqref{ISRWtoSSEPeq} to estimate
\begin{align}\label{pstochdomaeq1}
\mathbb{P}_{\nu_\rho}\left(\Upsilon_{r+1}(k,s) = 1 \mid \mathcal{F}_{r+1}^{s}\right)
&\le \sum_{(x,t) \in \mathcal{B}_{r+1}(k,s)\cap \Z^2} 
\mathbb{P}_{\nu_\rho}\left( \widehat{\Sigma}_r^{k,s}(s,t) < \rho_r \omega_r \mid \mathcal{F}_{r+1}^{k,s}\right) \nonumber \\
& \le C_1 6 \Delta_{r+1}^2 e^{-C_2 \sqrt{\omega_r}} =: p_{r+1},
\end{align}
which decays super-exponentially fast in $r$; in particular, $p_{r+1} < 1$ for large enough $r$. 
Since $\Upsilon_{r+1}(k,s) = 0$ if $\mathbf{\Lambda}_{r+1}(k,s)$ is $(r+1)$-rarefied,
\eqref{pstochdomaeq1} holds $\mathbb{P}_{\nu_{\rho}}$-a.s.. 

To conclude, fix $a \in \mathbf{B}_{r+1}(0,2\Delta_{r+1})$ and note that, 
by the definition of being locally spoiled, $\Upsilon_{r+1}(k,s)$ only 
depends on $\xi_{s-2\Delta_{r+1}}$ and on the graphical representation inside $\mathbf{B}_{r+1}(k,s)$.
Therefore, for fixed $s$, the collection
\begin{equation}\label{pstochdomaeq2}
\big(\Upsilon_{r+1}(k,s)\big)_{k \colon (k,s) \in I_a }
\end{equation}
is jointly independent under $\mathbb{P}_{\nu_\rho}(\cdot \mid \mathcal{F}_{r+1}^s)$.
Thus, by ordering any sequence $(k_i,s_i)\in I_a$, $i=1,\ldots,n$, such that $s_i \le s_j$ if $i\le j$, we see that,
by \eqref{pstochdomaeq1}, $(\Upsilon_{r+1}(k_i,s_i))_{i=1}^n$ can be progressively coupled in a monotone way to $n$ 
independent Bernoulli($p_{r+1}$) random variables. 
\end{proof}

\subsection{Proof of Proposition~\ref{mainprop}(b)}
\label{subsec:proofmainpropb}

In this section, we use the same proof strategy as in Section~\ref{subsec:proofmainpropa},
but the arguments will be technically much simpler.

Set $\epsilon_r := e^{-\Delta_r}$.
We call a point $(x,t) \in \Z \times \R$ $r$-\emph{stuck}
if both Poissonian clocks in the graphical representation that lie to the right and to the
left of $x$ fail to ring between times $t$ and $t+\epsilon_r$.
A subset of $\Z \times \R$ is called $r$-stuck if all its points with integer coordinates are $r$-stuck.
Note that $r$-turbulent blocks are not $r$-stuck.

Let $\Upsilon^{\mathrm{ns}}_r(k,s) := \mathbbm{1}_{\{ B_r(k,s) \text{ is not } r \text{-stuck} \}}$.
Set $I_{odd}:=\{(x,t)\in \Delta_r \Z^2 \colon\, x \Delta_r^{-1} \text{ is odd} \}$,
$I_{even}:=\{(x,t)\in \Delta_r\Z^2 \colon\, x \Delta_r^{-1} \text{ is even} \}$ and
$\mathcal{P}^{\mathrm{ns}}_r := \{I_{odd}, I_{even}\}$.

\begin{lemma}\label{lemma:stochdomb}
$\Upsilon^\mathrm{ns}_r$ is stochastically dominated by a $\tilde{p}_r$-PPS with
partition $\mathcal{P}^\mathrm{ns}_r$, where $\tilde{p}_r$ decays super-exponentially fast in $r$.
\end{lemma}
\begin{proof}
By the definition of being $r$-stuck, we have
\begin{equation}\label{pnobigturbeq1}
\mathbb{P}_{\nu_\rho} \left( (x,t) \text{ is } r\text{-stuck} \right) = e^{-2\epsilon_r}.
\end{equation}
Therefore, for any $(k,s) \in \Delta_r \Z^2$,
\begin{equation}\label{pnobigturbeq2}
\mathbb{P}_{\nu_\rho}\left( B_r(k,s) \text{ is not } r \text{-stuck}\right)
\le \Delta_r^2 (1-e^{-2\epsilon_r}) \le 2\Delta_r^2e^{-\Delta_r},
\end{equation}
i.e., for each $(k,s)$, $\Upsilon^{\mathrm{ns}}_r(k,s)$ is stochastically dominated by
a Bernoulli($\tilde{p}_r$) random variable, 
where $\tilde{p}_r := 2\Delta_r^2e^{-\Delta_r}$ decays super-exponentially fast in $r$.
Note that $\Upsilon^{\mathrm{ns}}_r(k,s)$ only depends on the graphical representation inside 
$B_r(k,s) \cup \{y_1,y_2\}\times[s-2\Delta_r,s+\Delta_r]$, where $y_1 := k- 5 \Delta_r -1$ and $y_2 := k+6\Delta_r$
are the sites on the spatial boundary of $B_r(k,s)$. Therefore
$(\Upsilon^{\mathrm{ns}}_r(k,s))_{(k,s) \in I}$ are jointly independent if $I \in \{I_{odd}, I_{even}\}$, 
which finishes the proof.
\end{proof}

\begin{proof}[Proof of Proposition~\ref{mainprop}(b)]
Let
\begin{equation}\label{pmainpropbeq1}
\Phi^{\mathrm{ns}}_r(\ell):= \sup_{w \in \mathcal{W}_{\ell}} \#\{r\text{-blocks which intersect } w \text{ and are not } r\text{-stuck}  \}.
\end{equation}
Since $\Phi^\mathrm{t}_r(\ell) \le \Phi^{\mathrm{ns}}_r(\ell)$, it is enough to prove that
\begin{equation}\label{pmainpropbeq2}
\lim_{r \to \infty} \limsup_{\ell \to \infty} \ell^{-1} \Delta^2_r \Phi^{\mathrm{ns}}_r(\ell) = 0.
\end{equation}
But, for $\kappa \in (0, (6\log(N_0))^{-1})$ and $r \le \lfloor \kappa \log(\ell) \rfloor$,
\begin{equation}\label{pmainpropbeq3}
\frac{1}{\ell} \Delta^2_r \Phi^{\mathrm{ns}}_r(\ell) \le \frac{1}{\ell} \sum_{k = r}^{\lfloor \kappa \log(\ell) \rfloor} \Delta^2_k \Phi^{\mathrm{ns}}_k(\ell),
\end{equation}
so \eqref{pmainpropbeq2} follows from Lemma~\ref{lemma:stochdomb} and Proposition~\ref{prop:PPSseq}.
\end{proof}

%%%%%%%%%%%%%%%%%% REFERENCES %%%%%%%%%%%%%%%%%%%%%%%%%%%%%%%%%%%

\end{document}